\documentclass[12pt]{article}  
  
\usepackage{latexsym,amsfonts,amsmath,amssymb,epsfig,  
tabularx,amsthm,dsfont,enumerate}  
  
\makeatletter  
\newcommand{\LeftEqNo}{\let\veqno\@@leqno}  
\makeatother  
  
\theoremstyle{definition}  
\newtheorem{thm}{Theorem}  
\newtheorem{rem}[thm]{Remark}

\newtheorem{cor}[thm]{Corollary}  
\newtheorem{example}[thm]{Example}

\newcommand{\calS}{{\mathbb{S}}}  
\newcommand{\calA}{{\mathbb{APP}}}  
\newcommand{\calI}{{\mathbb{I}}}  
\renewcommand\rho{\varrho}   
\newcommand\reals{\mathbb{R}}

\newcommand\N{\mathbb{N}}  
\newcommand\APP{{\rm APP}}  
\newcommand\INT{{\rm INT}}  
\newcommand\Z{\mathbb{Z}}

\newcommand\eps{\varepsilon}  
  
\newcommand\uset{{\mathfrak{u}}}

\newcommand{\e}{\varepsilon}  
\newcommand{\lall}{\Lambda^{\rm all}}  
\newcommand{\lstd}{\Lambda^{\rm std}}  
  
\newcommand{\widebar}[1]{\mbox{\kern1pt\hbox{\vbox{\hrule height 0.5pt \kern0.25ex  
        \hbox{\kern-0.05em \ensuremath{#1 }\kern-0.05em}}}}\kern-0.1pt}  
\newcommand\il{\left<}  
\newcommand\ir{\right>}

\newlength{\fixboxwidth}  
\title{Tractability of Multivariate Problems \\  
for Standard and Linear Information \\  
in the Worst Case Setting: Part I}  
  
\author{Erich Novak\footnote{This  
author was partially supported by the DFG-Priority Program 1324.},\\  
Mathematisches Institut, Universit\"at Jena\\  
Ernst-Abbe-Platz 2, 07743 Jena, Germany\\  
email: erich.novak@uni-jena.de\\  
\qquad  
Henryk Wo\'zniakowski\footnote{This author was partially  
supported by the National Science  
Foundation and by the National Science Centre, Poland,  
based on the decision DEC-2013/09/B/ST1/04275.}\\  
Department of Computer Science, Columbia University,\\  
New York, NY 10027, USA, and\\  
Institute of Applied Mathematics, University of Warsaw\\  
ul. Banacha 2, 02-097 Warszawa, Poland\\  
email:\ henryk@cs.columbia.edu}  
\begin{document}  
\maketitle  
\begin{abstract}  
We present a lower error bound for approximating linear multivariate  
operators defined over Hilbert spaces in terms of the error bounds   
for appropriately constructed linear functionals   
as long as algorithms use function values. Furthermore, some of  
these linear functionals have the same norm as the linear operators.  
We then apply this error bound for linear (unweighted) tensor products.  
In this way we use negative tractability results known for linear functionals  
to conclude the same negative results for linear operators.   
In particular, we prove that $L_2$-multivariate approximation defined   
for standard Sobolev space suffers the curse of dimensionality if   
function values are used although the curse is not present if linear   
functionals are allowed.  
\end{abstract}  
  
\newpage  
  
\section{Introduction}  
The understanding of the intrinsic difficulty of approximation of   
$d$-variate problems   
is a challenging problem especially when   
$d$ is large.  
We consider  
algorithms that approximate $d$-variate problems and use  
finitely many linear functionals: we compare the class $\lall$  
of arbitrary linear information functionals with the class  
$\lstd$ of information functionals that are given by function evaluations  
at single points.  
  
To find best algorithms for the class $\lall$ is usually much   
easier than for the class $\lstd$, in particular if the source space is a Hilbert space. 
This is especially the case   
for the worst case setting. The state of art may be found in \cite{NW12},  
where the reader may find a number of surprising results. For example,  
there are multivariate problems for which the best rate of convergence   
of algorithms using $n$ appropriately chosen linear functionals  
is $n^{-1/2}$ whereas for $n$ function values the best rate can be   
arbitrarily bad, i.e., like $1/\ln(\ln(\cdots\ln(n))))$, where the number   
of $\ln$ can be arbitrarily large, see \cite{HNV} which is also reported   
in \cite{NW12} pp. 292-304. Furthermore,  
the dependence on $d$ may be quite different  
for the linear and standard classes. There are examples   
of interesting multivariate problems for which the dependence on $d$   
\emph{is not} exponential for the class $\lall$,   
and \emph{is} exponential for the class $\lstd$.   
The exponential dependence on $d$ is called the \emph{curse of dimensionality}.  
On the other hand, for some other multivariate problems   
there is no difference between $\lall$ and $\lstd$. Examples can be found,  
in particular, in \cite{GW11} and \cite{NW08,NW10,NW12}.  
  
Tractability deals with how the intrinsic difficulty of a   
multivariate problem depends on $d$ and on $\e^{-1}$, where  
$\e$ is an error threshold.  
We would like to know when the curse of dimensionality holds and when   
we have a specific dependence on $d$ which is not exponential.  
There are various ways of measuring the lack of exponential   
dependence and that leads to different notions of tractability.  
In particular, we have polynomial tractability (PT) if   
the intrinsic difficulty   
is polynomial in both $d$ and $\e^{-1}$.  
We have quasi-polynomial tractability (QPT) if 
the intrinsic difficulty   
is at most proportional to $\e^{-\,t\,\ln\,d}$   
for some $t$ independent of $d$ and $\e$.    
  
Obviously, tractability may depend on which of the classes $\lstd$ or $\lall$ is used.  
Tractability results for $\lstd$ cannot be better than for $\lall$.  
The main question is for which multivariate problems they are more or   
less the same or for which multivariate problems   
they are essentially different.  
  
These questions   
were already addressed in \cite{NW08,NW10,NW12}.  
Still,  especially the worst case setting is not fully understood.  
We would like to get a better understanding  
how the power of the standard class $\lstd$   
is related to the power of the class $\lall$ of information.  
Ideally, we would like to characterize for which multivariate problems   
the classes $\lstd$ and $\lall$ lead  
to more or less the same tractability  
results and for which tractability results are essentially different.   
  
We plan to write a number of papers about this problem under the same title.  
We present the first part of this project. We restrict ourselves to  
linear multivariate problems defined as approximation of a linear  
continuous operator $S:F\to G$ for general Hilbert spaces $F$ and~$G$.  
Since we want to study the class $\lstd$ we need to assume that function   
values are well defined and they correspond to linear continuous functionals.  
This is equivalent to assuming that $F$ is a reproducing kernel Hilbert   
space.   
  
For the worst case setting and for the class $\lall$, it is known  
what is the best way to approximate $S$.   
The intrinsic difficulty of approximating $S$  
is defined as the \emph{information complexity}  
which is the minimal number of linear functionals which are needed to  
find an algorithm whose worst case error is at most~$\e\|S\|$.  
This depends on the eigenvalues of the operator $S^*S:F\to F$.  
For the class $\lstd$ the situation is much more complex and  
the \emph{information complexity}, which is now   
the minimal number of function values needed to get an error $\e\|S\|$,   
depends not only on the eigenvalues of $S^*S$.   
  
Our first result is the construction of continuous linear functionals $I$   
which are at most as hard to approximate as $S$ for the class $\lstd$.   
Furthermore, we characterize $I$ for which $\|I\|=\|S\|$. They are of the form  
$$  
I=\left<\cdot,S^*g\right>_F\ \ \ \ \mbox{with}\ \ \   
g=\lambda_1^{-1/2}S\eta,  
$$  
where $\lambda_1$ is the largest eigenvalue of $S^*S$ and  
$\eta$ of norm $1$ belongs to the eigenspace corresponding to $\lambda_1$.  
Hence, if $\lambda_1$ is of multiplicity $1$  then the choice of $g$ is   
essentially unique. If $\lambda_1$ is of multiplicity larger than $1$, then   
the choice of $g$ is not unique and   
may lead to trivial or hard linear functionals $I$.    
  
For $I$ with $\|I\|=\|S\|$, the information complexity of $I$   
for the class $\lstd$  
is at most equal to the information complexity of $S$. Hence,   
if $I$ is hard to approximate so is $S$.   
  
The essence of this result is that   
for approximation of linear functionals over some Hilbert spaces  
there is a proof technique which allows to find sharp error bounds.  
This proof technique was developed in \cite{NW01} and requires that  
the reproducing kernel of $F_1$ has a so called decomposable part.   
  
We verify how this lower bound on approximating $S$ works   
for linear $d$-folded (unweighted) tensor product problems.   
Then the corresponding linear functionals  $I$ are also $d$-folded  
tensor products. We then may apply the existing negative   
tractability results for $I$ and conclude the same   
negative tractability results for $S$.   
   
We illustrate our approach for a number of examples.   
In particular, we consider the Sobolev space  
   $F_d=F_1^{\otimes\,d}$ with the reproducing kernel  
$$  
K_d(x,t)=\prod_{j=1}^d(1+\min(x_j,t_j))\ \ \ \ \mbox{for all}  
\ \ x,t\in[0,1]^d,  
$$  
and $G_d=L_2([0,1]^d)$.   
  
Let $S=S_d:F_d\to G_d$ be any non-zero tensor product  
operator $S_d=S_1^{\otimes\,d}$ with $S_1:F_1\to G_1$.   
Let $\{\lambda_j\}$ be the ordered sequence of    
eigenvalues of $S_1^*S_1$.  
Let   
$$  
{\rm decay}_\lambda:=\sup\{\,r>0\, :\ \lim_{n\to\infty}n^r\lambda_n=0\,\}  
$$  
denote the polynomial decay of the eigenvalues $\lambda_n$.  
If the set of $r$ above is empty we set ${\rm decay}_\lambda=0$.  
  
Let $\calS=\{S_d\}_{d=1}^\infty$. It is known, see \cite{GW11}, that   
$\calS$ is quasi-polynomially tractable (QPT) for the class $\lall$   
iff $\lambda_2<\lambda_1$ and ${\rm decay}_{\lambda}>0$.   
Furthermore, if $\lambda_2$ is positive   
then $\calS$ is not polynomially tractable (PT).  
On the other hand, if $\lambda_2=\lambda_1$ then $\calS$ suffers   
from the curse of dimensionality for the class $\lall$ (and obviously  
   also for $\lstd$).    
  
For the class $\lstd$, assume   
without loss of generality that $\lambda_2<\lambda_1$.   
Let $\eta_1$ be a  normalized eigenfunction   
corresponding to the largest eigenvalue $\lambda_1$.   
We prove that   
$\calS$ suffers from the curse of dimensionality if  
\begin{equation}\label{iffiff}  
\eta_1\not=aK_1(\cdot,t)=a\,(1+\min(\cdot,t))\ \ \ \   
\mbox{for all}\ \ a\in\reals\ \   
\mbox{and}\ \ t\in[0,1].   
\end{equation}  
This means that we have the curse of dimensionality  
as long as the eigenfunction of $S_1$   
corresponding to the largest eigenvalue is not proportional to the univariate  
reproducing kernel with one argument fixed. We then verify that this assumption  
holds for multivariate approximation, i.e., for $S_1f=f$.  
This partially solves the open problem 131 from \cite{NW12} p. 361.  
  
We believe that the assumption~\eqref{iffiff} is also necessary   
for the curse. More generally, we believe that  
for $\eta_1=aK_1(\cdot,t)$ for some real $a$ and $t$ from the common   
domain of univariate functions, and of course for $\lambda_2<\lambda_1$ and  
${\rm decay}_{\lambda}>0$,    
we have QPT for the class $\lstd$ and this holds for any $K_1$.   
But this will be the subject   
of the next part of our project.  
  
In this paper we discuss only unweighted tensor products and that is why we   
   do not have polynomial tractability (PT) for problems with   
two positive eigenvalues. PT and other notions of tractability may hold  
if we consider weighted tensor products with sufficiently decaying weights.  
This will be also a subject of our next study.  
  
\section{Relation between Linear Functionals and Operators}  
  
Consider a continuous linear and non-zero operator $S:F\to G$,  
where $F$ is a reproducing kernel Hilbert space of real functions $f$ defined over   
a common domain $D\subset \reals^d$  
for some positive integer $d$, and $G$ is a Hilbert space.  
We approximate $S$ by algorithms~$A_n$ that use at most~$n$   
linear functionals. Without loss of generality we may assume that $A_n$ is linear,  
see e.g., \cite{NW08,TWW88}. That is,  
$$  
A_nf=\sum_{j=1}^nL_j(f)\, g_j  
$$  
for some $L_j\in F^*$ and $g_j\in G$.  
Using the same proof as in \cite{NW08} p. 345, we may also assume   
that $g_j =Sf_j$ for some $f_j\in F$.   
  
We consider two classes of linear functionals $L_j$'s:  
\begin{itemize}  
\item  the linear class of information $\lall$  
which consists of all continuous and linear functionals $L_j$'s, i.e.,   
$\lall=F^*$, and   
\item the standard class   
of information $\lstd$ which consists of function values,  
i.e., $L_j(f)=\il f,K(\cdot,t_j)\ir_F=f(t_j)$ for some $t_j\in D$,   
where $K:D\times D\to \reals$ is the  reproducing kernel of $F$.   
\end{itemize}  
  
The $n$th minimal (worst case) error of approximating $S$   
for the class $\Lambda\in\{\lstd,\lall\}$   
is defined as  
$$  
e_n(S,\Lambda)=\inf_{L_1,\dots,L_n\in\Lambda,\,g_1,\dots,g_n\in G}\   
\sup_{\|f\|_{F}\le1}\,\|Sf-A_nf\|_G.  
$$  
For $n=0$, we take $A_n=0$ and then we obtain the initial error which is  
independent of $\Lambda$ and given by  
$$  
e_0(S,\Lambda)=e_0(S)=\|S\|_{F\to G}.  
$$  
 
For the class $\lall$, it is well known that   
$\lim_{n\to\infty}e_n(S,\lall)=0$  
iff $S$ is compact.  
This is why we always may assume that $S$ is compact.  
Then it is known  
that the $n$th minimal errors depend on the eigenvalues   
of  
$$ 
W=S^*S:F\to F. 
$$ 
More precisely,   
let $(\lambda_j,\eta_j)$ be eigenpairs of $W$,  
$$  
W\eta_j=\lambda_j\eta_j,\ \ \ \mbox{with}\ \ \ \il \eta_j,\eta_k\ir_F=\delta_{j,k}  
\ \ \mbox{and}\ \ \lambda_1\ge\lambda_2\ge\cdots.  
$$  
Observe here that the $\lambda_j$ are uniquely defined,  
but the $\eta_j$ are not unique.  
Moreover, we formally define $\lambda_j=0$ if the dimension  
of $F$ is finite and  
$j$ is larger than this dimension.  
Then  
$$  
e_n(S,\lall)=\sqrt{\lambda_{n+1}}\ \ \ \ \mbox{for}\ \ n=0,1,\dots.  
$$  
Hence $e_0(S)=\|S\|_{F\to G}=\sqrt{\lambda_1}$,   
and since $S$ is non-zero we have $\lambda_1>0$.   
  
The situation is much more complicated for the class $\lstd$. Obviously,  
$$  
e_n(S,\lstd)\ge e_n(S,\lall)\ \ \ \ \mbox{for all}\ \ n=0,1,\dots,  
$$ but it is not clear when the sequences $e_n(S,\lstd)$   
and $e_n(S,\lall)$ behave similarly. There are many papers  
studying the powers of $\lstd$ and $\lall$.   
The state of art can be found in~\cite{NW12}.  
In this paper we continue this study and show that   
the sequence $e_n(S,\lstd)$ can behave   
quite differently than the sequence $e_n(S,\lall)$.   
This will be done by showing first that many continuous and linear functionals   
are at most as hard to approximate as $S$ and there are   
functionals for which  we can also match the initial error $e_0(S)$   
of $S$.  
  
More precisely, for any $g\in G$ with $\|g\|_G=1$ define   
$$  
I_gf=\il f,S^*g\ir_F\ \ \ \ \mbox{for all}\ \ f\in F.  
$$  
Note that $S^*g\in F$ and therefore $I_g$ is a continuous linear functional.  
The $n$th minimal error of approximating $I_g$   
is defined as for $S$, this time with $G$ replaced by $\reals$.  
Clearly, $e_0(I_g)=\|S^*g\|_F$.   
   
\begin{thm}\label{thm1}  
\quad  
  
For any $g\in G$ with $\|g\|_G=1$ we have  
$$  
e_n(I_g,\lstd)\le e_n(S,\lstd) \qquad \mbox{for all} \quad n \in \N_0 . 
$$  
Furthermore,   
$$  
e_0(I_g)= e_0(S) \quad \hbox{and} \quad \Vert g \Vert_G = 1  \qquad \mbox{iff} \qquad  
g=\lambda_1^{-1/2}S\eta,  
$$  
where $\eta$ is any element of norm $1$  
with $S^*S \eta = \lambda_1 \eta$.  
\qed   
\end{thm}  
\begin{proof}  
Take an arbitrary linear algorithm $A_nf=\sum_{j=1}^nf(t_j)\,Sf_j$   
for approximating $S$.   
Define   
$$  
B_nf=\sum_{j=1}^nf(t_j)\,\il f_j,S^*g\ir_F  
$$  
as a linear algorithm for approximating $I_g$. Then  
$$  
I_gf-B_nf=\il f-\sum_{j=1}^nf(t_j)\,f_j,S^*g\ir_F=  
\il Sf-\sum_{j=1}^nf(t_j)\,Sf_j,g\ir_G=  
\il Sf-A_nf,g\ir_G.  
$$  
Therefore  
$$  
|I_gf-B_nf|\le \|Sf-A_nf\|_G\,\|g\|_G=\|Sf-A_nf\|_G.  
$$  
Taking the supremum over the unit ball of $F$ and then the infimum over  
$t_j$'s and $f_j$'s, we conclude that   
$$  
e_n(I_g,\lstd)\le e_n(S,\lstd),  
$$  
as claimed.   
 
Let $m$ be the multiplicity of the largest eigenvalue $\lambda_1$, i.e.,  
$\mbox{span}(\eta_1,\eta_2,\dots,\eta_m)$ is the eigenspace of $W$ for  
the eigenvalue $\lambda_1$. 
Take now $g=\lambda_1^{-1/2}S\eta$ for any   
$\eta\in\mbox{span}(\eta_1,\eta_2,\dots,\eta_m)$  with $\|\eta\|_F=1$.  
Then we have $S^*g=\lambda_1^{-1/2}W\eta=  
\lambda_1^{1/2}\eta$ and  
$$  
e_0(I_g)=\|S^*g\|_F=\sqrt{\lambda_1}=\|S\|_{F\to G}=e_0(S).  
$$  
Furthermore,  
$$  
\|g\|_G^2=\il \lambda_1^{-1/2}S\eta,\lambda_1^{-1/2}S\eta\ir_G=\lambda_1^{-1}  
\il\eta,W\eta\ir_F=\lambda_1^{-1}\lambda_1\|\eta\|_F^2 =1.  
$$  
We need to show that $e_0(I_g)=e_0(S)$ holds only for such $g$.  
Take then any $g$ from $G$ such that $\|g\|_G=1$ and $\|S^*g\|_F=\lambda_1^{1/2}$.  
We can represent   
$$  
 g=\alpha\,S\eta+h,  
$$  
where $\eta\in\mbox{span}(\eta_1,\eta_2,\dots,\eta_m)$   
with $\Vert \eta \Vert_F = 1$,   
$\alpha \in \reals$ and $h$ is orthogonal to $S\eta_j$ for  
all $j=1,2,\dots,m$, i.e.,   
$$  
0=\left<h,S\eta_j\right>_G=\left<S^*h,\eta_j\right>_F\ \ \ \   
\mbox{for}\ \ j=1,2,\dots,m.  
$$   
Since $\|S\eta\|_G=\lambda_1^{1/2}$ we have   
$$  
1=\|g\|^2_G= \alpha^2 \lambda_1+\|h\|^2_G.  
$$  
On the other hand,   
\begin{eqnarray*}  
1=\frac1{\lambda_1}\|S^*g\|^2_F&=&\frac1{\lambda_1}\| \alpha \, S^*S\eta+S^*h\|^2_F  
=\frac1{\lambda_1} \|\alpha \lambda_1\,\eta+S^*h\|^2_F\\  
&=&  
\frac1{\lambda_1}\left(\alpha^2 \lambda_1^2+\|S^*h\|^2_F\right)=  
     \alpha^2 \lambda_1+\frac1{\lambda_1}\|S^*h\|^2_F.  
\end{eqnarray*}  
We now analyze $\|S^*h\|_F$. Note that  
$\|S^*h\|^2_F=\left<h,SS^*h\right>_G$. Let  
$$  
G=  \overline{SS^*(G)}\,\oplus\,[\overline{SS^*(G)}]^{\perp}.  
$$  
Hence, for any $h$ from $G$ we have $h=h_1+h_2$ with $h_1\in  
\overline{SS^*(G)}$ and   
$h_2$ orthogonal to $\overline{SS^*(G)}$. Then  
\begin{equation}\label{stuff}  
\|S^*h\|^2_F=\left<h_1+h_2,SS^*h_1\right>_G=\left<h_1,SS^*h_1\right>_G.  
\end{equation}  
Let $k=\sup\{j:\ \lambda_j>0\}$. Hence, if all $\lambda_j>0$  then $k=\infty$,  
otherwise $k$ is the number of positive eigenvalues $\lambda_j$.  
Clearly, $k\ge m$.  
  
We know that $S^*S\eta_j=\lambda_j\eta_j$. Then $S\eta_j\not=0$ for all  
finite $j$ which are at most $k$. Multiplying the last equation by $S$   
we obtain  
$$  
SS^*(S\eta_j)=\lambda_j(S\eta_j).  
$$  
Hence, $(\lambda_j, \lambda_j^{-1/2}S\eta_j)$ is an eigenpair of $SS^*$ and  
$$  
\left<\lambda_j^{-1/2}S\eta_j,\lambda_i^{-1/2}S\eta_i\right>_G  
=\frac{\left<S\eta_j,S\eta_i\right>_G}{\sqrt{\lambda_j\lambda_i}}=  
\frac{\left<\eta_j,S^*S\eta_i\right>_F}{\sqrt{\lambda_j\lambda_i}}=  
\frac{\lambda_i}{\sqrt{\lambda_j\lambda_i}}\left<\eta_j,\eta_i\right>_F=  
\delta_{i,j}.  
$$  
That means that $\lambda_j^{-1/2}S\eta_j$'s are orthonormal in $G$.   
Since $SS^*(G) \subset S(F)$, the  
$\lambda_j^{-1/2}S\eta_j$'s  
build a complete orthonormal system of $\overline{SS^*(G)}$ 
and, when we return to \eqref{stuff}, we may write 
$$  
h_1=\sum_{j=1}^k  
\left<h_1,\lambda_j^{-1/2}S\eta_j\right>_G\lambda_j^{-1/2}S\eta_j   
$$  
and then  
$$  
\left<h_1,SS^*h_1\right>_G=\sum_{j=1}^k\left<h_1,  
\lambda_j^{-1/2}S\eta_j\right>_G^2  \lambda_j.  
$$  
Since $0=\left<h,S\eta_j\right>_G=\left<h_1,S\eta_j\right>_G$ for all $j\le m$,  
we conclude that   
\begin{eqnarray*}  
\left<h_1,SS^*h_1\right>_G&=&\sum_{j=m+1}^k  
\left<h_1,\lambda_j^{-1/2}S\eta_j\right>_G^2 \lambda_j\le  
\lambda_{m+1}\sum_{j=m+1}^k   
\left<h_1,\lambda_j^{-1/2}S\eta_j\right>_G^2 \\  
&=&\lambda_{m+1}\|h_1\|^2_G  
\le \lambda_{m+1}\|h\|^2_G.   
\end{eqnarray*}  
{}From this, we get  
\begin{eqnarray*}  
1=\alpha^2\lambda_1+\frac1{\lambda_1} \|S^*h\|^2_F&\le&  
\alpha^2\lambda_1  
+\frac{\lambda_{m+1}}{\lambda_1}\,\|h\|^2_G  
=  
 \alpha^2\lambda_1+\|h\|^2_G-  
\left(1-\frac{\lambda_{m+1}}{\lambda_1}\right)\,\|h\|^2_G\\  
&=&  
1- \left(1-\frac{\lambda_{m+1}}{\lambda_1}\right)\,\|h\|^2_G.  
\end{eqnarray*}  
Since $\lambda_{m+1}<\lambda_1$   
we conclude that $h=0$ and $g=\alpha\,S\eta$ with $\alpha^2\lambda_1=1$.  
This completes the proof.    
\end{proof}      
  
For any $g$ from $G$ 
the linear functional $I_g$ can be also written as  
$$  
I_gf=\il Sf,g\ir_G\ \ \ \ \mbox{for all}\ \ f\in F.  
$$  

\begin{example} 
As an example, if we take $G=L_2([0,1]^d)$ then  
$$  
I_gf=\int_{[0,1]^d}(Sf)(x)\,g(x)\,{\rm d}x.  
$$  
Furthermore, if we additionally assume that $F$ is continuously   
embedded in $L_2([0,1]^d)$ and take $S=\APP_d$ as multivariate approximation,  
$\APP_d  f=f$ for all $f\in F$, then  
$$  
I_gf=\int_{[0,1]^d}f(x)\,g(x)\,{\rm d}x.  
$$   
If $g\equiv 1$ then   
$$  
\INT_df=I_gf=\int_{[0,1]^d}f(x)\,{\rm d}x  
$$  
is multivariate integration, and   
$$  
e_0(\INT_d)\le e_0(\APP_d).  
$$  
This relation between multivariate integration and approximation has  
been used in many papers.  
For some spaces the norm of multivariate integration and approximation is   
the same. This is the case for Korobov spaces and some Sobolev spaces  
as will be reported later.  
  
However, in general, the norm of multivariate integration is smaller  
and sometimes exponentially smaller   
than the norm of multivariate approximation.  
This is the case for some other Sobolev spaces. For instance, this holds  
for the space $F=F_d$ with the reproducing kernel  
$$  
K_d(x,y)=\prod_{j=1}^d(1+\min(x_j,y_j))\ \ \ \  
\mbox{for all}\ \ x_j,y_j\in[0,1].  
$$  
It is known, see \cite{NW12} pp. 353 and 411, that   
$$  
e_0(\INT_d)=(4/3)^{\,d/2}=(1.3333\dots)^{\,d/2}  
\ \ \ \mbox{and}\ \ \ e_0(\APP_d)=(1.35103388\dots)^{\,d/2} . 
$$  
Hence,   
$$  
\frac{e_0(\APP_d)}{e_0(\INT_d)}=(1.013\dots)^{\,d/2}.  
$$  
Although $1.013\dots$ is barely larger than one, the ratio  
of the initial errors for multivariate approximation and integration  
goes to infinity exponentially fast with $d$.  
   
As we shall see,   
the multiplicity of the largest eigenvalue   
for this   
multivariate approximation is $m=1$. Therefore,   
in order to match the norm of multivariate approximation   
we must use   
a weighted integration problem   
$I_g(f) = (f,g)_{L_2}$ with   
$g=\lambda_1^{-1/2}\eta_1$ (or $g=-\lambda_1^{-1/2}\eta_1$)   
which for our example of $F_d$   
is not equal to the constant function $1$.   
In Section \ref{unweighted} we will show that $g(x)=\prod_{j=1}^dg_1(x_j)$   
for $x=[x_1,\dots,x_d]\in[0,1]^d$.  
We find it interesting to know the ``most difficult''  
integration problem $I_g$ (with $\Vert g \Vert_2=1$) for a Hilbert space of functions  
and hence present  
the graph of the function $g_1$   
in Figure 1.   
The same $g$ is also the unique (up to a multiplicative constant)  
function that maximizes  
$\frac{\Vert g \Vert_2}{\Vert g \Vert_F}$  and hence solves an important  
optimization problem.  
\begin{figure}[ht]\begin{center}  
\includegraphics[width=10cm]{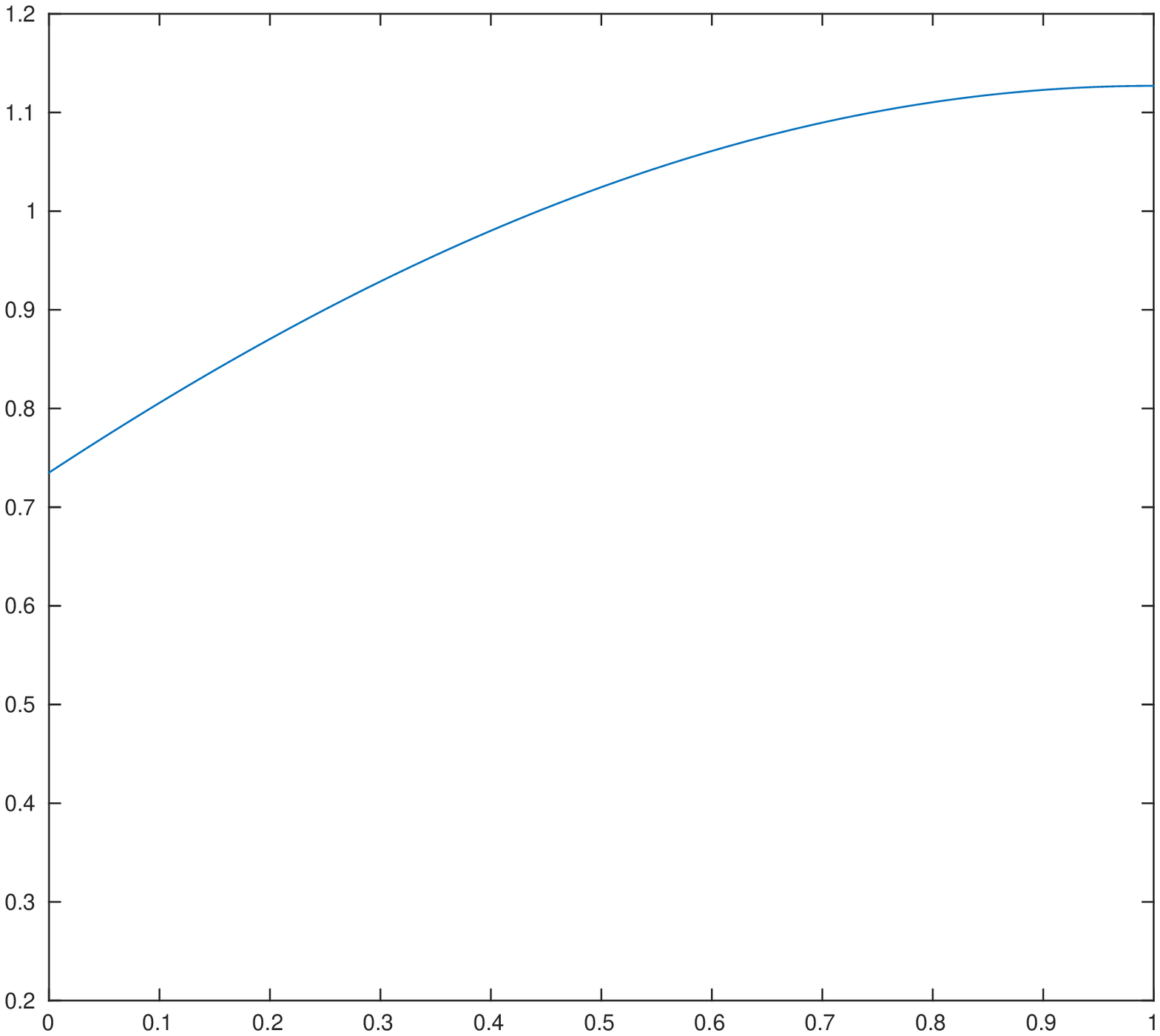}  
\vspace*{-.5cm}  
\caption{Density $g$ for which $\Vert I_g \Vert = \Vert \APP_1 \Vert$}   
\label{figure}   
\end{center}\end{figure}  
\end{example} 
  
We now show that the choice of $\eta$ in $g=\lambda_1^{-1/2}S\eta$   
may be important if the multiplicity  
of~$\lambda_1$ is larger than $1$. That is, it may happen that for some such   
$g$ the functional $I_g$ is trivial and for some other $g$,   
it may be very difficult.   
 
\goodbreak  
  
\begin{example}  
\quad  
 
Let $F_1$ be the space of  
functions $f: [0,1] \to \reals$ 
that are constant over $[0,\tfrac12]$ and $(\tfrac12,1]$. 
That is, for $f\in F_1$ there are real $a$ and $b$ such that  
$f(x)=a$ for all $x\in[0,\tfrac12]$  
and $f(x)=b$ for all $x\in(\tfrac12,1]$.   
 
We equip $F_1$ with the $L_2$ norm which can be written  
(for the space $F_1$) as  
$$  
\|f\|_{F_1}=\tfrac1{\sqrt{2}}\,\left(f^2(0)+f^2(1)\right)^{1/2}.  
$$  
We define $F=G=F_1^{\otimes d}$ as the $d$ folded tensor product of the   
space $F_1$. 
The space $F$ consists of piecewise constant functions  
over $2^d$ subintervals of volume $2^{-d}$ which   
are a partition of the cube $[0,1]^d$. The space $F_d$ is also equipped  
with the $L_2$ norm. Clearly, $\mbox{dim}(F)=2^d$.  
  
Let $S:F\to G$ be the identity operator. Then $\lambda_j=1$ for  
all $j=1,2,\dots,2^d$ and any nonzero function from $F$ is   
an eigenfunction of $W=S^*S=I$. Clearly, $\lambda_{2^d+1}=0$ and therefore  
$$  
e_n(S,\lall)=1 \ \ \ \mbox{for all}\ \ n<2^d\ \ \mbox{and}\ \   
e_{2^d}(S,\lall)=0.  
$$  
Obviously,   
it also proves that $e_n(S,\lstd)=1$ for $n<2^d$   
since  $1=e_n(S,\lall)\le e_n(S,\lstd)\le e_0(S)=1$.  
  
For any $g$ of norm $1$ we have $e_0(I_g)=e_0(S)=1$. We now show that   
$e_n(I_g,\lstd)$ very much depends on the choice of $g$.    
  
Suppose we take $g=2^{d/2}$ over $[0,\tfrac12]^d$  
and $g=0$ otherwise. Then  
$$  
I_g(f)=2^{d/2}\,\int_{[0,\tfrac12]^d}f(x)\,{\rm d}x=2^{-d/2}\,f(0)  
$$   
is a trivial linear functional which can be solved exactly by using   
one function value at~$0$. Hence,  
$$  
e_n(I_g,\lstd)=0\ \ \ \ \mbox{for all}\ \ n\ge1.  
$$  
In this case, the bound $0=e_n(I_g,\lstd)\le e_n(S,\lstd)=1$ is useless.    
  
Take now $g=1$ over the cube $[0,1]^d$. Then  
$$  
I_g(f)=\int_{[0,1]^d}f(x)\,{\rm d}x=\frac1{2^d}\,\sum_{j=[j_1,j_2,\dots,j_d]\in  
\{0,1\}^d}f(t_{j_1},t_{j_2},\dots,t_{j_d}),  
$$  
where $t_0=0$ and $t_1=1$. We prove that the $n$th minimal error  
for $n\le 2^d$ is   
$$  
e_n(I_d,\lstd)=\left(1-n\,2^{-d}\right)^{1/2}.  
$$  
Indeed, for $n=2^d$ we can sample $f$ at all points $f(t_1,t_2,\dots,t_d)$   
with $t_j\in\{0,1\}$ and recover~$I_g$ exactly. Therefore $e_{2^d}(I_d,\lstd)=0$.  
  
Assume now that $n<2^d$. Suppose we sample $f$ at some $x_1,x_2,\dots,x_n$  
from the unit cube $[0,1]^d$.    
Then it is enough to take $f$ which is zero at $n$ sub-cubes that  
contain samples  
$x_1,x_2,\dots,x_n$, and which takes a constant value $c$   
at $2^d-n$ sub-cubes.   
Taking $c$ for which the norm is $1$ we obtain   
the equation $1=c^2(2^d-n)/2^d$ which yields $c=\pm\,(2^d/(2^d-n))^{1/2}$. Then     
$$  
I_g(f)=c\,\frac{2^d-n}{2^d}=\pm\,\left(1-n\,2^{-d}\right)^{1/2}.  
$$  
All linear algorithms must approximate $I_g(f)$ by zero and therefore   
their worst case error is at least $\left(1-n\,2^{-d}\right)^{1/2}$.  
The last bound is sharp if we take sample points $x_1,x_2,\dots,x_n$ at disjoint  
sub-cubes, as claimed.    
  
Hence, in this case we have  
$$  
\left(1-n\,2^{-d}\right)^{1/2}=e_n(I_g,\lstd)\le e_n(S,\lstd)=1.  
$$  
The bound is quite sharp as long as $n$ is much smaller than $2^d$.   
   
\qed  
\end{example}   
 
For general spaces, we will use   
Theorem~\ref{thm1} for $I_g$ with $g=\lambda_1^{-1/2}S\eta_1$.  
For the standard class of information~$\lstd$,  
using lower bounds results for $I_g$   
from \cite{NW10},  
we obtain lower bounds results for $S$.   
In this way we show, in particular, that we have sometimes   
the curse of dimensionality for~$\lstd$ which is not present   
for the class $\lall$.    
  
\section{Tractability Notions}  
 
We need to recall the definition of the information complexity for  
the so-called normalized error criterion.  It is defined as  
the minimal number of linear functionals from the class  
$\Lambda$  which are needed to reduce  
the initial error by  a factor $\e\in(0,1)$, where   
$\Lambda\in\{\lstd,\lall\}$. That is,    
$$  
n(\eps,S,\Lambda)=\min\{\,n\,:\ e_n(S,\Lambda)\le \eps\,e_0(S)\}.  
$$  
For the class $\lall$, we obviously have  
$$  
n(\e,S,\lall)=\min\{\,n\,:\ \lambda_{n+1}\le \e^2\,\lambda_1\}.  
$$  
Unfortunately, there is no such or similar formula for the class $\lstd$.   
   
Assume now that we have a sequence   
$$  
\calS=\{S_d\}_{d=1}^\infty  
$$   
of continuous linear non-zero operators   
$S_d:F_d\to G_d$, where $F_d$ is a reproducing kernel Hilbert space   
of real function defined over $D_d\subset \reals^d$ and $G_d$   
is a Hilbert space.   
In this case, we want to verify  
how the information complexity $n(\e,S_d,\Lambda)$ depends on $\e^{-1}$  
and $d$.    
  
In this paper we will use only a few tractability notions which are defined  
as follows. We say that   
\begin{itemize}  
\item $\calS$ suffers from the curse of dimensionality   
for the class $\Lambda\in\{\lstd,\lall\}$   
iff there are positive numbers $c$ and $C$ as well as   
$\e_0\in(0,1)$ such that  
$$  
n(\e,S_d,\Lambda)\ge C\,(1+c)^d  
$$  
for all $\e\in(0,\e_0]$ and infinitely many $d$.  
\item   
$\calS$ is quasi-polynomially tractable (QPT) for the class $\Lambda$   
iff there are non-negative numbers $C$ and $t$ such that  
$$  
n(\e,S_d,\Lambda)\le C\,\exp\left(\,t\,  
(1+\ln\,\e^{-1})(1+\ln\,d)\right)\ \ \ \   
\mbox{for all}\ \ \e\in(0,1), \, d=1,2,\dots.  
$$  
The infimum of numbers $t$ satisfying the bound above is called   
the exponent of QPT and denoted by $t^*$.  
\item   
$\calS$ is polynomially tractable (PT) for the class $\Lambda$  
iff there are nonnegative numbers $C,p,q$ such that  
$$  
n(\e,S_d,\Lambda)\le C\,\e^{-p}\,d^{\,q}\ \ \ \   
\mbox{for all}\ \ \e\in(0,1),\, d=1,2,\dots.    
$$  
\end{itemize}  
Clearly, PT implies QPT.  
More about these and other tractability concepts   
can be found in \cite{NW08,NW10,NW12}.  
   
For the class $\lall$, tractability notions depend on the decay   
of the eigenvalues $\lambda_{d,j}$ of the operator $W_d=S^*_dS_d:F_d\to F_d$.   
Necessary and sufficient conditions can be found in the works cited above.  
Again for the class $\lstd$, no such conditions are known  
and they cannot depend only on the eigenvalues $\lambda_{d,j}$.  
  
\section{Linear Tensor Product Problems}  
 
{}From now on we study a sequence  
$$  
\calS=\{S_d\}_{d=1}^\infty  
$$   
of tensor product problems.  
Hence  the spaces $F=F_d = F_1^{\otimes d}$ and $G=G_d = G_1^{\otimes d}$  
as well as $S=S_d = S_1^{\otimes d}$ are given by tensor products of 
$d$ copies of $F_1$  
and $G_1$ as well as a continuous linear operator $S_1:F_1\to G_1$,  
respectively,   
where $F_1$ is a reproducing kernel Hilbert space of real univariate   
functions defined over $D_1\subset \reals$ and $G_1$ is a Hilbert space.   
Then $F_d$ is a space of $d$-variate real functions defined   
on $D_d=D_1\times D_1\times\cdots\times D_1$ ($d$ times). 

An important example 
is given by multivariate approximation.  
That is, we now take $G_d=L_2([0,1]^d)$ and   
$S_1=\APP_1:F_1\to G_1$ with the embedding operator   
$\APP_1f=f$. Then $S_d=\APP_d:F_d\to G_d$ is also the embedding  
operator $\APP_df=f$ for all $f\in F_d$. In this case, we denote   
$$  
\calS=\calA, \ \ \ \mbox{where}\ \ \ \calA=\{\APP_d\}_{d=1}^\infty.  
$$  

If $K_1$ is   
   the  reproducing kernel of $F_1$ then $F_d$ is a reproducing kernel Hilbert space   
whose kernel is  
$$  
K_d(x,y)=\prod_{j=1}^dK_1(x_j,y_j)\ \ \ \   
\mbox{for all}\ \ x=[x_1,x_2,\dots,x_d],\,y=[y_1,y_2,\dots,y_d]\in D_d.  
$$  
For the class $\lall$, it is well  
known that the eigenpairs $(\lambda_{d,j},\eta_{d,j})$ of $W_d=S_d^*S_d:F_d\to F_d$   
are given in terms of the eigenpairs $(\lambda_j,\eta_j)$ of the univariate   
operator $W_1=S^*_1S_1:F_1\to F_1$. As before we assume   
that $\lambda_1\ge\lambda_2\ge\cdots$,  $\il \eta_j,\eta_k\ir_{F_1}=  
\delta_{j,k}$ and that $S_1$ is non-zero. This means that $\lambda_1>0$.   
For the operator $W_d$ we have   
$$  
\{\lambda_{d,j}\}_{j=1}^\infty=\{\lambda_{j_1}\lambda_{j_2}\cdots\lambda_{j_d}  
\}_{j_1,j_2,\dots,j_d=1}^\infty.  
$$  
Similarly, the eigenfunctions of $W_d$ are of product form  
$$  
\{\eta_{d,j}\}_{j=1}^\infty=\{  
\eta_{j_1}\otimes  
\eta_{j_2}\otimes \cdots\otimes \eta_{j_d}  
\}_{j_1,j_2,\dots,j_d=1}^\infty,     
$$  
where  
$$  
[\eta_{j_1}\otimes  
\eta_{j_2}\otimes \cdots\otimes \eta_{j_d}](x)=  
\prod_{k=1}^d\eta_{j_k}(x_k)\ \ \ \   
\mbox{for all}\ \ x=[x_1,\dots,x_d]\in D_d.  
$$  
Then $\|S_d\|_{F_d\to G_d}=\|W_d\|_{F_d\to F_d}^{1/2}=\lambda_1^{d/2}$.  
Hence, the initial error is  
$$  
e_0(S_d)=\lambda_1^{d/2}.  
$$  
If $\lambda_2=\lambda_1$ then we have at least $2^d$ eigenvalues of $W_d$   
equal to $\lambda_1^d$, and therefore   
$$  
n(\e,S_d)\ge2^d\ \ \ \ \mbox{for all}\  \ \e\in(0,1)\ \ \mbox{and}\ \   
d=1,2,\dots\,.  
$$  
In this case $\calS=\{S_d\}_{d=1}^\infty$   
suffers from the curse of dimensionality.  
On the other hand, it is proved in  
\cite{GW11}, see also \cite{NW12} p. 112,  
that $\calS$ is QPT for the class $\lall$ iff  
$\lambda_2<\lambda_1$ and   
\begin{equation}\label{decay}  
{\rm decay}_\lambda:=\sup\{\,r>0\, :\ \lim_{n\to\infty}n^r\lambda_n=0\,\}>0.  
\end{equation}  
If the last conditions hold then the exponent of QPT is  
$$  
t^*=\max\left(\frac2{{\rm decay}_\lambda},\frac2{\ln\,  
\frac{\lambda_1}{\lambda_2}}\right).  
$$  
Note that for $\lambda_2=0$ we have $t^*=0$. In this case, $S_d$ is a   
continuous linear functional and $n(\e,S_d,\lall)=1$  
for all $\e\in[0,1)$ and all $d=1,2,\dots\,$. If $\lambda_2>0$ then  
the exponent of QPT is positive and in this case it is also known that  
the problem $\calS$ is \emph{not} PT for the class $\lall$.  
  
We now turn to the class $\lstd$. Without loss of generality we assume  
that $\lambda_2<\lambda_1$ since otherwise $\calS$   
suffers from the curse of dimensionality also for the class $\lstd$.   
Then the choice of the element $g$   
for which $e_0(I_g)=e_0(S)$ in Theorem \ref{thm1}   
is essentially unique  
and we take $g=g_d$ with $g_d=\lambda_{d,1}^{-1/2}S_d\eta_{d,1}$.  
We have     
$S^*_dg_d=\lambda_{d,1}^{1/2}\,\eta_{d,1}$ which is a tensor product since   
$$  
S_d^*g_d=(\lambda_1^{1/2} \eta_1)\otimes\cdots \otimes (\lambda_1^{1/2}\eta_1)  
\ \ (\mbox{$d$ times}).  
$$  
Let   
$$  
I_1f=\il f,\lambda_1^{1/2}\eta_1\ir_{F_1}\ \ \ \ \mbox{for all}  
\ \ f\in F_1.  
$$  
Then   
$$  
I_d=I_1\otimes\cdots\otimes I_1\ \ (\mbox{$d$ times})  
$$  
is a linear tensor product functional. We have  $I_d=I_{g_d}$  
and $e_0(I_d)=e_0(S_d)$.   
Let  
$$  
\calI=\{I_d\}_{d=1}^\infty.  
$$  
Theorem~\ref{thm1} yields that  
$$  
n(\e,I_d,\lstd)\le n(\e,S_d,\lstd)\ \ \ \   
\mbox{for all}\ \ \e\in(0,1),\ d=1,2,\dots\,.  
$$  
This implies the following corollary  
\begin{cor}\label{cor111}  
\quad  
  
If one of the tractability notions does not hold for $\calI$  
then it also does not hold for $\calS$.   
\qed  
\end{cor}  
  
We now illustrate Corollary~\ref{cor111} for two examples for which $\calI$  
is multivariate integration  and for which  
it is known that multivariate integration   
suffers from the curse of dimensionality.  
  
\begin{example}\label{kor} {\bf Korobov Space}  
\quad  
  
As in \cite{GW11,NW12}, let $F_1$   
be a Korobov space whose reproducing kernel is  
$$  
K_1(x,y)=1+2\beta\,\sum_{k=1}^\infty\frac{\cos(2\pi(x-y))}{k^{2\alpha}}  
\ \ \ \ \mbox{for all}\ \ x,y\in[0,1]   
$$  
for some $\beta\in(0,1]$ and $\alpha>\tfrac12$. This corresponds to the norm  
$$  
\|f\|_{F_1}^2=|\widehat{f}(0)|^2+\beta^{-1}\,\sum_{h\in\Z\setminus\{0\}}  
|h|^{2\alpha}\,|\widehat{f}(h)|^2,  
$$  
for Fourier coefficients $\widehat{f}(h)$ of $f$. 
We take $G_1=L_2([0,1])$   
and $S_1f=f$, hence we consider the approximation problem $\calA$. 

In this case we know that   
$$  
\lambda_1=1, \ \ \lambda_2=\beta\ \ \ \mbox{and}\ \ \   
\eta_1\equiv 1.  
$$  
Hence, for $\beta=1$ multivariate approximation suffers   
from the curse of dimensionality for $\lall$ (and of course for $\lstd$),  
and for $\beta<1$, multivariate approximation is QPT with the exponent  
$$  
t^*=\max\left(\frac1{\alpha},\frac2{\ln\,\beta^{-1}}\right).  
$$  
  
Note that $g_1=\lambda_1^{-1/2}S_1\eta_1\equiv 1$. Therefore  
$g_d\equiv 1$ and   
$$  
I_d(f)=\int_{[0,1]^d}f(x)\,{\rm d}x\ \ \ \ \mbox{for all}  
\ \ f\in F_d  
$$  
is multivariate integration.   
{}From Theorem 16.16 on p. 457 in \cite{NW10}  
 which is based on \cite{HW01,NW01},  
we know that multivariate integration suffers from the curse of dimensionality.  
So does multivariate approximation   
due to Corollary~\ref{cor111}.   
\qed  
\end{example}  
  
\begin{example} {\bf Sobolev Space}   
\quad  
 
We now take the Sobolev space $F_1$ of absolutely   
continuous functions on $[0,1]$  whose first derivatives  
are square integrable with the inner product  
$$  
\left<f,u\right>_{F_1}=\int_0^1f(x)\,u(x)\,{\rm d}x\ +\ \int_0^1  
f^{\prime}(x)\,u^{\prime}(x)\,{\rm d}x.  
$$  
This space has the intriguing reproducing kernel  
$$  
K_1(x,t)=\frac1{\sinh(1)}\,\cosh(1-\max(x,t))\,\cosh(\min(x,t))  
\ \ \ \ \mbox{for all}\ \ x,t\in[0,1],  
$$  
see \cite{BT04}.  
We consider the $L_2$  approximation problem, as in Example 4.  
Hence we have $S_1 f =f$ and $G_1 = L_2([0,1])$.  
  
In this case we have for $d=1$ that  
$\lambda_1=1$ is of multiplicity $1$, and $\eta_1\equiv 1$.   
The second largest eigenvalue satisfies the condition  
$$  
\lambda_2=\max_{f\in F_1,\, \int_0^1f(x)\,{\rm dx}=0}  
\frac{\int_0^1f^2(x)\,{\rm d}x}{\int_0^1f^2(x)\,{\rm d}x+ \int_0^1  
[f^{\prime}(x)]^2\,{\rm d}x}=\frac1{1+\mu_2},  
$$  
where  
$$  
\mu_2=\min_{f\in F_1,\,\int_0^1f(x)\,{\rm  
    d}x=0}\frac{\int_0^1[f^\prime(x)]^2\,{\rm d}x}{\int_0^1f^2(x)\,{\rm  
    d}x}.  
$$  
It is known, see e.g. \cite{PW60}, that $\mu_2=\pi^2$. Hence,  
have   
$$  
\lambda_2=\frac1{1+\pi^2}=0.091999668\dots.  
$$  
It is well known that $\lambda_j=\Theta(j^{-2})$ so that ${\rm  
decay}_\lambda=2$.   
This implies that multivariate approximation for tensor products $F_d=F_1^{\otimes\,d}$  
and $G_d=L_2([0,1]^d)$ is QPT for $\lall$ with the exponent  
$$  
t^*=1.   
$$  
Due to the form of $(\lambda_1,\eta_1)$, the linear functional  
$I_d$ corresponds to multivariate integration. It is known that  
multivariate integration suffers from the curse, see \cite{SW02}  
which is also reported in \cite{NW10} pp. 605-606.  
Hence, multivariate approximation also suffers the curse of dimensionality   
for the class $\lstd$ due to Corollary~\ref{cor111}.  
\qed  
\end{example}  
  
Tractability of tensor product functionals $\calI$    
was thoroughly studied   
in \cite{NW01}, see also Chapters 11 and 12 of \cite{NW10}.  
In particular, for many spaces $F_1$ the problem $\calI$  
suffers from the curse of dimensionality for the class $\lstd$.   
This holds if the reproducing kernel $K_1$ of $F_1$ has a decomposable part  
and the univariate function $\eta_1$ has non-zero components   
with respect to the decomposable part.  
If this is the case then $\calS$ also suffers from the curse of dimensionality   
for the class $\lstd$ although we may have QPT for the class $\lall$.   
We will mention more specific results in the next section.   
  
\section{Sobolev Space}\label{unweighted}  
  
We now consider tensor product problems $\calS$ defined as in the previous section    
for the space $F_1$ taken as a Sobolev space of univariate   
real functions defined over $[0,1]$.   
More precisely, let $F_1$ be the space   
of absolutely continuous functions defined over $[0,1]$ and whose  
first derivatives belong to $L_2([0,1])$. The space $F_1$ has the   
reproducing kernel  
\begin{equation}\label{ker1}  
K_1(x,y)=1+\min(x,y)\ \ \ \ \mbox{for all}\ \ x,y\in [0,1],  
\end{equation}  
and the inner product for $f,h\in F_1$ is   
$$  
\il f,h\ir_{F_1}=f(0)\,h(0)+\int_0^1f^{\prime}(x)\,h^{\prime}(x)\,{\rm d}x.  
$$  
For the tensor product space $F_d = F_1^{\otimes d}$ of $d$ copies of $F_1$,   
the inner product for $f,h\in F_d$ is now of the form  
$$  
\il f,h\ir_{F_d}=f(0)\,h(0)+\sum_{\emptyset\not=\uset\subseteq\{1,2,\dots,d\}}  
\int_{[0,1]^{|\uset|}}\frac{\partial^{|\uset|}f}{\partial x_{\uset}}(x_\uset,0)  
\frac{\partial^{|\uset|}h}{\partial x_{\uset}}(x_\uset,0)\,{\rm d}x_{\uset},  
$$  
where $\partial x_{\uset}=\prod_{j\in\uset}\partial x_j$, ${\rm d}x_{\uset}=  
\prod_{j\in \uset}{\rm d}x_j$ and   
$(x_{\uset},0)$ is a $d$ dimensional vector with   
components~$x_j$ for $j\in \uset$ and $0$ otherwise.   
  
It was proved in \cite{NW10} pp.195-200, see also \cite{NW10a},   
that for any linear   
non-zero tensor product functional   
its information complexity (for $\Lambda^{\rm std}$)  is $1$ or it is  
exponentially large in $d$. Furthermore, the information complexity  
is $1$ only for trivial cases when the linear tensor   
product functional is of the form   
$$  
a^df(t,t,\dots,t)=\il f, h_d\ir_{F_d}  
\ \ \ \mbox{with}\ \ \ h_d(x)=\prod_{j=1}^da\,(1+\min(x_j,t))  
$$   
for some non-zero real $a$ and for some $t\in [0,1]$.   
Applying this results for $\calI$ we see that as long as   
\begin{equation}\label{goodcase}  
\eta_1\not=a(1+\min(\cdot,t))\ \ \ \ \mbox{for all\, $a\in\reals\   
\mbox{and}\ t\in[0,1]$}  
\end{equation}  
then $\calI$ as well as $\calS$ suffer from the curse of dimensionality   
for the class $\lstd$. We summarize the results from the last two sections  
in the following theorem.  
\begin{thm}\label{thm2}  
\quad  
  
Consider a linear non-zero tensor product problem $\calS$,  
as  defined in this section.   
\begin{itemize}  
\item Let $\lambda_2=\lambda_1$.   
Then $\calS$ suffers   
from the curse of dimensionality for $\lall$ (and $\lstd$).  
\item Let $\lambda_2<\lambda_1$. Then    
$\calS$ is QPT for $\lall$ iff \eqref{decay} holds.  
\item Let $0<\lambda_2<\lambda_1$. Then      
$\calS$ is not PT for the class $\lall$.  
\item Let $\lambda_2<\lambda_1$. If \eqref{goodcase}   
holds then $\calS$ suffers from the curse of dimensionality for $\lstd$.  
\end{itemize}   
\qed  
\end{thm}  
  
In general, the assumption~\eqref{goodcase}   
used in the last part of    
Theorem \ref{thm2} is   
needed.   
Indeed, if \eqref{goodcase} does not hold then we may have  
$S_d$ as a linear tensor product functional of the form   
$S_df=a^df(t,t,\dots,t)$ with a nonzero real $a$ and $t\in[0,1]$.  
Then $\calS$ is trivial since $n(\e,S_d,\lstd)=1$ for all $\e\in[0,1)$  
and all $d$.  
  
We now verify the assumptions \eqref{decay} and \eqref{goodcase}  
for multivariate approximation $\calA$. 
    
The eigenpairs $(\lambda_j,\eta_j)$ were found in \cite{WW95}, see also  
\cite{NW12} pp. 409-411. We have  
$\lambda_j=\alpha_j^{-2}$, where  
$\alpha_j\in((j-1)\pi,j\pi)$ is the unique solution of the nonlinear equation  
$$  
\cot\,x=x\ \ \ \ \ \mbox{for}\ \ x\in ((j-1)\pi,j\pi),  
$$  
and   
$$  
\eta_j(x)=\beta_j\,\cos (\alpha_jx-\alpha_j)\ \ \ \ \mbox{for all}\ \ x\in[0,1],  
$$     
where  
$$  
\beta_j=\left(\cos^2(\alpha_j)+\tfrac{\alpha}2\,   \left(\alpha_j  
           -\tfrac12\sin(2\alpha_j)\right)\right)^{-1/2} .  
$$  
For $j=1$ and $j=2$ the numerical computation yields  
\begin{eqnarray*}  
\lambda_1&=&1.35103388\dots,\\  
\lambda_2&=&0.08521617\dots \,.  
\end{eqnarray*}  
Clearly, $\alpha_j=\pi\,j\,(1+o(1))$ as $j$ tends to infinity. This shows that  
$$  
\lambda_j=\frac{1+o(1)}{\pi^2\,j^2}\ \ \ \ \mbox{as}\ \ j\to\infty.  
$$  
Therefore ${\rm decay}_{\lambda}=2$. Hence \eqref{decay} holds  
and $\calA$ is QPT for the class $\lall$.   
  
The assumption \eqref{goodcase} holds if for all $a\in\reals$ and $t\in[0,1]$  
we can find $x\in[0,1]$ such that   
$$  
\cos(\alpha_1x-\alpha_1)\not= a(1+\min(x,t)).  
$$  
For $t=0$, the right hand side is constant, whereas the left hand side varies.  
For $t>0$, the right hand side is constant over $[0,t]$, whereas  
the left hand side varies. Therefore \eqref{goodcase} also holds.   
We summarize this in the following corollary.  
\begin{cor}\label{cor1}  
\quad  
  
Consider the multivariate approximation problem $\calA$ for the Sobolev spaces studied 
   in this section. 
Then  
\begin{itemize}  
\item $\calA$ is QPT but not PT for the class $\lall$ with the exponent of QPT  
$$  
t^*= 1.   
$$  
\item $\calA$ suffers from the curse of dimensionality for the class $\lstd$.  
\end{itemize}  
\qed  
\end{cor}    
We add in passing that a similar analysis can be done   
if the reproducing kernel~\eqref{ker1} is replaced by  
$$  
K_{1,a}(x,t)=1+\tfrac12\left(|x-a|+|t-a|-|x-t|\right)\ \ \ \   
\mbox{for all}\ \ x,t\in[0,1],  
$$  
for any $a\in[0,1]$ or by   
$$  
K_1(x,t)=\min(x,t)\ \ \ \ \mbox{for all}\ \ x,t\in[0,1].  
$$  
For these variants of the Sobolev spaces,  
Corollary~\ref{cor1} is valid.   
 
\begin{rem}  
We conclude this paper with a comment on the rates of convergence and 
tractability notions for $\Lambda^{\rm all}$ and $\Lambda^{\rm std}$.  
In \cite{HNV},  
the $L_2$ approximation problem was studied. It was shown that   
there are classes $F$ for which the best rate of convergence   
of algorithms using $n$ appropriately chosen linear functionals  
is $n^{-1/2}$ whereas for $n$ function values the best rate can be   
arbitrarily bad. 
If the best rate for $\Lambda^{\rm all}$ is faster than  
$n^{-1/2}$  
than we still do not know whether the rates for $\Lambda^{\rm std}$  
and $\Lambda^{\rm all}$ always coincide.  
 
For the examples in this section the rates for $\Lambda^{\rm std}$ 
and $\Lambda^{\rm all}$ are basically (up to log terms)~$n^{-1}$  
but tractability properties for  
$\Lambda^{\rm all}$ and $\Lambda^{\rm std}$  
are quite different. Hence, even if the rates are the 
same,  tractability properties can be quite different for 
$\Lambda^{\rm all}$ and $\Lambda^{\rm std}$.   
\end{rem}  
 
\noindent  
{\bf Acknowledgement.}  
We thank Mario Ullrich who computed some numbers for us.  
We also thank Greg Wasilkowski,  Markus Weimar and  
two referees  for valuable comments.


\begin{thebibliography}{99}  
  
\setlength{\parsep }{-0.5ex}  
\setlength{\itemsep}{-0.5ex}  
  
\frenchspacing  
  
\newcommand\BAMS{\emph{Bull. Amer. Math. Soc.\ }}  
\newcommand\BIT{\emph{BIT\ }}  
\newcommand\Com{\emph{Computing\ }}  
\newcommand\CA{\emph{Constr. Approx.\ }}  
\newcommand\FCM{\emph{Found. Comput. Math.\ }}  
\newcommand\JAT{\emph{J. Approx. Th.\ }}  
\newcommand\JC{\emph{J. Complexity\ }}  
\newcommand\JMA{\emph{SIAM J. Math. Anal.\ }}  
\newcommand\JMAA{\emph{J. Math. Anal. Appl.\ }}  
\newcommand\JMM{\emph{J. Math. Mech.\ }}  
\newcommand\MC{\emph{Math. Comp.\ }}  
\newcommand\NM{\emph{Numer. Math.\ }}  
\newcommand\RMJ{\emph{Rocky Mt. J. Math.\ }}  
\newcommand\SJNA{\emph{SIAM J. Numer. Anal.\ }}  
\newcommand\SR{\emph{SIAM Rev.\ }}  
\newcommand\TAMS{\emph{Trans. Amer. Math. Soc.\ }}  
\newcommand\TOMS{\emph{ACM Trans. Math. Software\ }}  
\newcommand\USSR{\emph{USSR Comput. Maths. Math. Phys.\ }}  
  
\frenchspacing  
\bibitem{BT04}  
A. Berlinet and C. Thomas-Agnan,  
\emph{Reproducing Kernel Hilbert Spaces in  
Probability and Statistics},   
Kluwer Academic Publishers, Boston, 2004.   
  
\bibitem{GW11}  
M. Gnewuch and H. Wo\'zniakowski,  
Quasi-polynomial tractability, \JC 27, 312-330, 2011.  
  
\bibitem{HW01}  
F. J. Hickernell and H. Wo\'zniakowski,  
Tractability of multivariate integration for periodic functions,  
\JC 20, 660-682, 2001. 
  
\bibitem{HNV}  
A. Hinrichs, E. Novak and J. Vybiral,  
Linear information versus function evaluations for $L_2$-approximation,  
\JAT 153, 97-107, 2008.  
  
\bibitem{NW01}  
E. Novak and H. Wo\'zniakowski,  
Intractability results for integration and discrepancy,  
\JC 17, 388-441, 2001.  
  
\bibitem{NW08}  
E. Novak and H. Wo\'zniakowski,  
\emph{Tractability of Multivariate Problems},  
Volume I: Linear Information,  
European Math. Soc. Publ. House, Z\"urich,  
2008.  
  
\bibitem{NW10}  
E. Novak and H. Wo\'zniakowski,  
\emph{Tractability of Multivariate Problems},  
Volume II: Standard Information for Functionals,  
European Math. Soc. Publ. House, Z\"urich,  
2010.  
  
\bibitem{NW10a}  
E. Novak and H. Wo\'zniakowski,  
Tractability of approximating multivariate linear functionals.  
In: The Steve Smale Festschrift,  
\emph{J. Fixed Point Th. Appl.}, 7, 313-324, 2010.  
  
\bibitem{NW12}  
E. Novak and H. Wo\'zniakowski,  
\emph{Tractability of Multivariate Problems},  
Volume III: Standard Information for Operators,  
European Math. Soc. Publ. House, Z\"urich,  
2012.  
  
\bibitem{PW60}  
L. E. Payne and H. F. Weinberger,  
An optimal Poincar\'e inequality for convex domains.  
\emph{Arch. Rational Mech. Anal.}, 5, 286-292, 1960.  
  
\bibitem{SW02}  
I. H. Sloan and H. Wo\'zniakowski,  
Tractability of integration in non-periodic and periodic  
weighted tensor product Hilbert spaces,  
\JC 18, 479-499, 2002.   
  
  
\bibitem{TWW88}  
J. F. Traub, G. W. Wasilkowski and H. Wo\'zniakowski,  
\emph{Information-Based Complexity},  
Academic Press, 1988.  
  
\bibitem{WW95}  
G. W. Wasilkowski and H. Wo\'zniakowski,  
Weighted tensor-product algorithms for linear   
multivariate problems,  
\JC 15, 402--447, 1999.   
  
\end{thebibliography}
\end{document}